\documentclass[12pt]{amsart}

\usepackage{amsmath}
\usepackage{amsfonts}
\usepackage{amssymb}
\usepackage{amsthm}
\usepackage{hyperref}
\usepackage{enumerate}
\usepackage{cleveref}
\usepackage{mathrsfs}
\usepackage[top=0.9in, bottom=1.15in, left=1.1in, right=1.1in]{geometry}
 \usepackage{hyperref}
\hypersetup{colorlinks=true,linkcolor=blue,citecolor=magenta}
\newtheorem{theorem}{Theorem}
\newtheorem{lemma}[theorem]{Lemma}

\newtheorem{corollary}[theorem]{Corollary}
\newtheorem{definition}[theorem]{Definition}
\newtheorem*{remark}{Remark}

\Crefname{conjecture}{Conjecture}{Conjectures}

\theoremstyle{plain}
\newtheorem{example}[theorem]{Example}
\theoremstyle{plain}

\author{Robert Schneider}
\thanks{This material is based upon work partially supported by the National Science Foundation under grant DMS-1128155.}
\address{Department of Mathematics and Computer Science\newline
Emory University\newline
400 Dowman Dr., W401\newline
Atlanta, Georgia 30322}
\email{robert.schneider@emory.edu}

\begin{document}

\title
[Jacobi's triple product, mock theta functions, unimodal sequences]
{Jacobi's triple product, mock theta functions,\\unimodal sequences and the $q$-bracket}


\begin{abstract}
In Ramanujan's final letter to Hardy, he listed examples of a strange new class of infinite series he called ``mock theta functions''. It turns out all of these examples are essentially specializations of a so-called universal mock theta function $g_3(z,q)$ of Gordon--McIntosh. Here we show that $g_3$ arises naturally from the reciprocal of the classical Jacobi triple product---and is intimately tied to rank generating functions for unimodal sequences, which are connected to mock modular and quantum modular forms---under the action of an operator related to statistical physics and partition theory, the $q$-bracket of Bloch--Okounkov. Secondly, we find $g_3(z,q)$ to extend in $q$ to the entire complex plane minus the unit circle, and give a finite formula for this universal mock theta function at roots of unity, that is simple by comparison to other such formulas in the literature; we also indicate similar formulas for other $q$-hypergeometric series. Finally, we look at interesting ``quantum'' behaviors of mock theta functions inside, outside, and on the unit circle. 
\end{abstract}

\maketitle

\section{Introduction: Sea Change}\label{Sect0}

When Ramanujan put to sea 
from India for Cambridge University in 1914, destined to revolutionize number theory, a revolution in physics was already full-sail in Europe. Just one year earlier, the Rutherford--Bohr model of atomic shells heralded the emergence of a paradoxical new {\it quantum theory} of nature that contradicted common sense. 
In 1915, Einstein would describe how space, light, matter, geometry itself, warp and bend in harmonious interplay. 
The following year, 
Schwarzschild found Einstein's equations to predict the existence of monstrously inharmonious 
{\it black holes}, that 
we can now study directly (just very recently) using interstellar gravitational waves \cite{Gravwaves}. 

During Ramanujan's five years working with G. H. Hardy, 
news of the paradigm shift in physics must have created a thrill among the mathematicians at Trinity College, Isaac Newton's {\it alma mater}. 
Certainly Hardy would have been aware of the sea change. After all, J. J. Thomson's discovery of the electron, as well as his subsequent ``plum-pudding'' atomic model, had been made at Cambridge's Cavendish Laboratory; Rutherford 
had done his post-doctoral work with Thomson there; 
and Niels Bohr came to Cambridge to work under Thomson in 1911 \cite{Gamow}. 
Moreover, Hardy's 
intellectual colleague David Hilbert was in a public race with Einstein to write down the equations of General Relativity \cite{Race}. 

We don't know how aware Ramanujan was of these happenings in physics, 
yet his flights of imagination and break with academic tradition were expressions of the scientific {\it Zeitgeist} of the age. After returning to India in 1919, as he approached his own tragic event horizon, 
Ramanujan's thoughts ventured into realms that---like the domains of subatomic particles and gravitational waves---would require the technology of a future era to navigate \cite{PhelanSchneider}.

In the final letter he sent to Hardy, dated 1920, Ramanujan described a new class of mathematical objects he called {\it mock theta functions} \cite{RamanujanCollected}, that mimic certain behaviors of classical {\it modular forms} \cite{Apostol, Ono_web}.
 These $q$-series turn out to have 
profoundly curious analytic, combinatorial and algebraic properties. Ramanujan gave a prototypical example $f(q)$ of a mock theta function, defined by the $q$-series
\begin{equation}\label{f(q)def}
f(q):=\sum_{n=0}^{\infty}\frac{q^{n^2}}{(-q;q)_n^{2}},
\end{equation}  
where $|q|<1$ and the $q$-Pochhammer symbol is defined by $(z;q)_0:=1,\  (z;q)_n:=\prod_{0\leq i \leq n-1}(1-zq^i)$, with $(z;q)_{\infty}:=\lim_{n\to \infty}(z;q)_n$. Ramanujan claimed that $f(q)$ is ``almost'' modular in a number of ways. For instance, he provided a pair of 
functions $\pm b(q)$ with
$$b(q) := {(q;q)_{\infty}}{(-q;q)_{\infty}^{-2}}$$
that are modular up to multiplication by $q^{-1/24}$ when $q:=e^{2\pi \text{i} \tau}$, $\tau\in \mathbb H$ (the upper half-plane), to compensate for the singularities arising in the denominator of (\ref{f(q)def}) as $q$ approaches an even-order root of unity $\zeta_{2k}$ (where we define $\zeta_m:=e^{{2 \pi \text{i}}/{m}}$) radially from within the unit circle:
\begin{equation}\label{Watson}
\lim_{q\to \zeta_{2k}}\left(f(q)-(-1)^k b(q) \right)=\mathcal O (1).
\end{equation} 
This type of behavior was first rigorously investigated by Watson in 1936 \cite{Watson_final}, and quantifies to some degree just how ``almost'' modular $f(q)$ is: at least at even-order roots of unity, $f(q)$ looks like a modular form plus a constant. 

Only in the twenty-first century have we begun to grasp the larger meaning of functions such as this (see \cite{OnoVisions}), 
in the wake of Zwegers's innovative Ph.D. thesis \cite{Zwegers} of 2002. We now know Ramanujan's mock theta functions are examples of {\it mock modular forms}, which are the holomorphic parts of more complicated creatures called {\it harmonic Maass forms} \cite{OnoVisions}. 
In 2012, Folsom--Ono--Rhoades \cite{FolsomOnoRhoades} made explicit the limit in (\ref{Watson}), showing that 
\begin{equation}\label{Watson2}
\lim_{q\to \zeta_{2k}}\left(f(q)-(-1)^k b(q) \right)=-4U(-1,\zeta_{2k}),
\end{equation} 
where $U(z,q)$ is the {\it rank generating function for strongly unimodal sequences} in combinatorics (see \cite{BFR}), and is closely related to {\it partial theta functions} and mock modular forms. 
By this connection to $U$, the work of Folsom--Ono--Rhoades along with Bryson--Ono--Pitman--Rhoades \cite{BOPR} reveals that the mock theta function $f(q)$ is strongly connected to the newly-discovered species of {\it quantum modular forms} in the sense of Zagier: functions that are modular up to the addition of some ``suitably nice'' function, and (in Zagier's words) have ``the `feel' of the objects occurring in perturbative quantum field theory''\cite{Zagierquantum}. 


We do not know what sparked Ramanujan to discover mock theta functions,  
but we see in this paper that they are indeed natural forms to study from a classical perspective. It turns out in Section \ref{Sect1} that all of his mock theta functions---to be precise, the odd-order {\it universal mock theta function} of Gordon--McIntosh that essentially specializes to the mock theta functions Ramanujan wrote down \cite{GordonMcIntosh}---arise from the {\it Jacobi triple product}, a fundamental object in number theory and combinatorics \cite{Berndt}, and are generally ``entangled'' with rank generating functions for unimodal sequences, under the action of the {\it $q$-bracket} operator from statistical physics and partition theory that boils down to multiplication by $(q;q)_{\infty}$ \cite{Zagier}. 
In Section \ref{Sect2} we find finite formulas for the odd-order universal mock theta function 
and indicate similar formulas for other $q$-hypergeometric series. In some cases,  
we do encounter the ``feel'' of quantum theory.  
   
\section{Connecting the triple product to mock theta functions}\label{Sect1}

At the wildest boundaries of nature, 
we see tantalizing hints of $q$-series. Recall that Borcherds proposed a proof of the Jacobi triple product formula based on the {\it Dirac sea} model of the quantum vacuum (see \cite{Cameron}); it is as if this beautiful, versatile identity emerges from properties of empty space. 
Also from the universe of $q$-hypergeometric series, mock theta functions and their generalization mock modular forms \cite{OnoVisions} are connected conjecturally 
to deep mysteries in physics, like 
mind-bending phenomena at the edges of black holes \cite{DMZ,DGO}. 
All the diversity of physical reality---and of our own mental experience---plays out quite organically between these enigmatic extremes. 
Perhaps not unrelatedly, in this study we see there is an organic 
connection between the Jacobi triple product and mock theta functions. 


Let us fix notations and concepts. 
We take $q,z\in \mathbb C, |q|<1, z\neq 0$ throughout, unless stated otherwise. 
The {\it Jacobi triple product} $j(z;q)$  is implicit in countless famous classical identities (see \cite{Berndt}). Up to multiplication by rational powers of $q$, the triple product specializes to the Jacobi theta function (that Ramanujan constructed ``mock'' versions of), a weight $1/2$ modular form which is also important in physics as a solution to the heat equation. The triple product is defined by 
\begin{equation}\label{jtp}
j(z;q):=(z;q)_{\infty}(z^{-1}q;q)_{\infty}(q;q)_{\infty}.
\end{equation}
We note the simple zero at $z=1$ from the $(1-z)$ factor in $(z;q)_{\infty}$.

The odd-order {\it universal mock theta function} $g_3(z,q)$ of Gordon and McIntosh \cite{GordonMcIntosh}, which specializes to Ramanujan's original list of mock theta functions up to changes of variables and multiplication by rational powers of $q$ and $z$, is defined as
\begin{equation}\label{umtf}
g_3(z,q):=\sum_{n=1}^{\infty} \frac{q^{n(n-1)}}{(z;q)_{n}(z^{-1}q;q)_{n}},
\end{equation}
and, like the triple product, is subject to all sorts of wonderful transformations. We note the simple pole at $z=1$. 

Let us recall that a {\it unimodal sequence} of integers is of the type 
\begin{equation*}
0\leq a_1\leq a_2\leq...\leq a_r \leq \overline{c}\geq b_1\geq b_2\geq...\geq b_s\geq 0.
\end{equation*}
The term $\overline{c}$ is called the {\it peak} of the sequence; generalizing this concept, if $\overline{c}$ occurs with multiplicity $\geq k$, we might consider the unimodal sequence with a {\it $k$-fold peak}  
\begin{equation*}
0\leq a_1\leq a_2\leq...\leq a_r \leq \overline{c}\  \overline{c}\  ...\  \overline{c} \geq b_1\geq b_2\geq...\geq b_s\geq 0,
\end{equation*}
where ``$\overline{c}\  \overline{c}\  ...\  \overline{c}$'' denotes $k$ repetitions of $\overline{c}$. 
When all the inequalities above 
are strictly ``$<$'' or ``$>$'' the sequence is {\it strongly unimodal}. 

If $r$ is the number of $a_i$ to the left and $s$ is the number of $b_j$ to the right of a unimodal sequence, the difference 
$s-r$ is called the {\it rank} of the sequence; and the sum of all the terms including the peak is the {\it weight} of the sequence. Another series that plays a role here is the {\it rank generating function $\widetilde{U}(z,q)$ for 
unimodal sequences}, 
given by 
\begin{equation}\label{U}
\widetilde{U}(z,q):=\sum_{n=0}^{\infty}\frac{q^n}{(zq;q)_n (z^{-1}q;q)_n}= \sum_{n=0}^{\infty}\sum_{m=-\infty}^{\infty}\widetilde{u}(m,n)z^m q^n,
\end{equation}
where $\widetilde{u}(m,n)$ is the number of unimodal sequences of rank $m$ and weight $n$. Each summand of the first infinite series is the generating function for unimodal sequences with peak term $n$: the factor $(z^{-1}q;q)_n^{-1} $ generates $a_i\leq n$, $(zq;q)_n^{-1}$ generates $b_j \leq n$ and the $q^n$ factor inserts $n$ as the peak term $\overline{c}$ (following \cite{BOPR,KimLovejoy}). If we replace $z$ with $-z$, the right-most series is actually the very first expression Andrews revealed from Ramanujan's ``lost'' notebook (\cite{AndrewsLost}, Eq. 1.1) shortly after unearthing the papers at Trinity College \cite{SchneiderLost}. 
This form, which is related to partial theta functions \cite{KimLovejoy}, was swimming alongside mock theta functions in the Indian mathematician's imagination during his final year. 

Let $\mathcal P$ denote the set of integer partitions $\lambda = (\lambda_1, \lambda_2,...,\lambda_r)$ for $\lambda_i \in \mathbb Z^+$, $\lambda_1 \geq \lambda_2\geq ...\geq \lambda_r\geq 1$. We let $\ell (\lambda):=r$ be the {\it length} of $\lambda$, let $|\lambda|$ denote the {\it size} (i.e., the number being partitioned), and take ``$\lambda \vdash n$'' to mean $\lambda$ is a partition of $n\geq 1$.  Finally, following Bloch--Okounkov \cite{BlochOkounkov} as well as Zagier \cite{Zagier}, we define the {\it $q$-bracket} $\left<f\right>_q$ of a function $f:\mathcal P \to \mathbb C$ 
to be given by
\begin{equation}\label{qbracket}
\left<f\right>_q:\  =\  \frac{\sum_{\lambda\in \mathcal P}f(\lambda)q^{|\lambda|}}{\sum_{\lambda\in \mathcal P} q^{|\lambda|}}\  =\  (q;q)_{\infty}\sum_{\lambda\in \mathcal P}f(\lambda)q^{|\lambda|},
\end{equation}
where the sums are taken over all partitions. This operator represents the expected value in statistical physics of a measurement over a grand ensemble whose states are indexed by partitions with weights $f$, for a canonical choice of $q$; this is the content of the quotient in the middle of (\ref{qbracket}).

However, we proceed formally here using the right-most expression, 
without drawing too much physical interpretation (while always keeping the mysterious feeling that our formulas resonate in physical reality). 
Simply multiplying by $(q;q)_{\infty}$ induces quite interesting $q$-series phenomena:
Bloch--Okounkov \cite{BlochOkounkov}, Zagier \cite{Zagier}, and Griffin--Jameson--Trebat-Leder \cite{GJTL} show that the $q$-bracket can produce families of modular, quasimodular and $p$-adic modular forms; and the present author finds the $q$-bracket to play a natural role in partition theory as well \cite{Robert_bracket, Tanay}, modularity aside. (We highly recommend Zagier's paper \cite{Zagier} for more about the $q$-bracket.) 

We will see here that the reciprocal of the Jacobi triple product 
\begin{equation*}
j(z;q)^{-1}=:\sum_{\lambda \in \mathcal P} j_z(\lambda) q^{|\lambda|}
\end{equation*}
has a very rich and interesting interpretation in terms of the $q$-bracket operator, which (multiplying $j(z;q)^{-1}$ by $(q;q)_{\infty}$) has the shape  
\begin{equation*}
\left<j_z\right>_q=\frac{1}{(z;q)_{\infty}(z^{-1}q;q)_{\infty}}.
\end{equation*} 
Note that this $q$-bracket also has a simple pole at $z=1$. We abuse notations somewhat in writing the coefficients $j_z$ in this way, as if $z\in \mathbb C$ were a constant. 
In fact, 
$j_z$ is a map from the partitions to $\mathbb Z[z]$, which the author finds in \cite{Robert_bracket} to be given explicitly by
\begin{equation}
j_z(\lambda)=(1-z)^{-1}\sum_{\delta|\lambda}\sum_{\varepsilon|\delta}z^{\operatorname{crk}(\varepsilon)}
\end{equation}
for $z\neq1$, where ``$\alpha | \beta$'' means $\alpha\in\mathcal P$ is a {\it subpartition} of $\beta\in\mathcal P$ (all the parts of $\alpha$ are also parts of $\beta$), and ``$\operatorname{crk}$'' denotes the {\it crank} statistic of Andrews--Garvan \cite{AndrewsGarvan} defined below.\footnote{The methods of \cite{Robert_bracket} will yield a slightly more complicated formula for the coefficients of $j(z;q)$ itself.} 

\begin{definition}\label{crk}
The {\it crank} $\operatorname{crk}(\lambda)$ of a partition $\lambda$ is equal to its largest part 
if the multiplicity $m_1(\lambda)$ of 1 as a part of $\lambda$ is $=0$ (that is, there are no $1$'s), and if $m_1(\lambda)> 0$ then $\operatorname{crk}(\lambda)=\#\{\text{parts of $\lambda$ that are larger than $m_1(\lambda)$\}}-m_1(\lambda).$
\end{definition}

\begin{remark}
The well-known {\it crank generating function} (see \cite{AndrewsGarvan}) is given by
\begin{equation*}
C(z;q)=\frac{(q;q)_{\infty}}{(zq;q)_{\infty} (z^{-1}q;q)_{\infty}}=(1-z)(q;q)_{\infty}\left<j_z\right>_q.
\end{equation*} 
Expanding $C(z;q)$ as a power series in $q$, the $n$th coefficient is given by $\sum_{\lambda \vdash n}z^{\operatorname{crk}(\lambda)}$. 
\end{remark}


In \cite{Robert_bracket} we used the $q$-bracket operator to find the coefficients 
of $\left<j_z\right>_q$ explicitly in terms of sums over subpartitions and the crank statistic, as well. Now we take a different approach, and look at $\left<j_z\right>_q$ from the point-of-view of $q$-hypergeometric relations. It turns out the odd-order universal mock theta function $g_3$ (in an ``inverted'' form) and the unimodal rank generating function $\widetilde{U}$ naturally arise together as components of $\left<j_z\right>_q$.

\begin{theorem}\label{theorem1} 
For $0<|q|<1, z\neq 0,z\neq 1$, the following statements are true:
\begin{enumerate}[(i)]
            \item \label{thm1.1}
            We have the $q$-bracket formula
            \begin{equation*}\left<j_z\right>_q=1\  +\  \left[z(1-q)+z^{-1}q\right] g_3(z^{-1},q^{-1})\  +\  \frac{zq^2}{1-z}\widetilde{U}(z,q).\end{equation*}
 
            \item The ``inverted'' mock theta function component in part (\ref{thm1.1}) converges, and can be written in the form
\begin{equation*}\label{umtfsum}
g_3(z^{-1},q^{-1})=\sum_{n=1}^{\infty}\frac{q^n}{(z;q)_n (z^{-1}q;q)_n}.
\end{equation*}
        \end{enumerate}

\end{theorem}
\  

By considering the factor $z(1-q)+z^{-1}q$ as $|z|\to \infty$ and as $|z|\to 0$ in part (\ref{thm1.1}) of Theorem \ref{theorem1}, we get the following asymptotics.

\begin{corollary}\label{cor1}
We have the asymptotic estimates: 
\begin{enumerate}[(i)]
\item For $0<|q|<1 \ll |z|$, we have
\begin{equation*}
\left<j_z\right>_q\sim z(1-q) g_3(z^{-1},q^{-1})\  \text{as}\  |z|\to \infty.
\end{equation*}
\item For $0< |q|<1,\  0<|z|\ll 1$, we have 
\begin{equation*}
\left<j_z\right>_q\sim z^{-1}q \  g_3(z^{-1},q^{-1})\  \text{as}\  |z|\to 0.
\end{equation*}
\end{enumerate}
\end{corollary}
\  

Thus the inverted mock theta function component 
dominates the behavior of the $q$-bracket for $z$ not close to the unit circle (which is ``most'' of the complex plane). 

\begin{remark}
So the universal mock theta function is the main influence on these expected values for large and small $|z|$, with appropriate choice of $q$.
\end{remark}

Conversely, if we write 
\begin{equation*}
\left<j_z\right>_q=:\sum_{n=0}^{\infty}c_n q^n,\  \  \  \  g_3(z^{-1},q^{-1})=:\sum_{n=0}^{\infty}\gamma_n q^n,
\end{equation*} 
where the coefficients $c_n=c_n(z),\  \gamma_n=\gamma_n(z)$ also depend on $z$, 
then the author proves an explicit combinatorial formula for the $c_n$ in \cite{Robert_bracket} using nested sums over subpartitions of $n$, viz.
\begin{equation}\label{coeffs}
c_n(z)={(1-z)^{-1}}\sum_{\lambda\vdash n}\sum_{\delta | \lambda}\sum_{\epsilon | \delta} \sum_{\varphi | \epsilon}\mu(\lambda / \delta)z^{\operatorname{crk}(\varphi)},
\end{equation}
where, in the notations of \cite{Robert_bracket}, $\mu(\alpha)=0$ if $\alpha\in\mathcal P$ has any part repeated and $=(-1)^{\ell (\alpha)}$ otherwise (a partition-theoretic version of the M\"{o}bius function first considered by Alladi \cite{Alladi2015, Robert_bracket}), 
and if $\alpha | \beta$ then ``$\beta \slash \alpha$'' denotes the partition obtained by deleting the parts of $\alpha$ from $\beta$.

With (\ref{coeffs}) in hand, it follows from Corollary \ref{cor1} that the coefficients of $g_3(z^{-1},q^{-1})$ satisfy the asymptotic 
\begin{equation}
\gamma_n(z) \sim \left\{
        \begin{array}{ll}
            z^{-1} (c_1+c_2+...+c_n)\  \text{as}\ |z|\to \infty\\
            \  \\  
            z c_{n-1}\   \text{as}\  |z|\to 0,\  n\geq 1
        \end{array}
    \right.
\end{equation} 
(which depends entirely on the growth of $z$, not $n$), as the coefficients enjoy the recursion
\begin{equation*}
\gamma_n-\gamma_{n-1}\sim z^{-1}c_n\  \text{for}\  |z|\gg 1.
\end{equation*}

It is a well-known fact (see, for instance, \cite{Ono_web}) that if $\zeta_*\neq 1$ is a root of unity, then $$(\zeta_* q;q)_{\infty} (\zeta_*^{-1} q;q)_{\infty} $$ is, up to multiplication by a rational power of $q$, a modular function; 
but this product is the reciprocal of $$(1-\zeta_*)\cdot \left<j_z\right>_q \bigr|_{z=\zeta_*}.$$ This is another example of the intersection of the $q$-bracket with modularity phenomena, and at the same time gives a feeling for the obstruction to the inverted mock theta function's sharing in this modularity at $z=\zeta_*$; for $g_3(z^{-1},q^{-1})$ is not necessarily a dominating aspect of $\left<j_z\right>_q$ for $z\neq 1$ near the unit circle, whereas the unimodal rank generating aspect $\widetilde{U}(z,q)$ makes a more noticeable contribution, and the two pieces work together to produce modular behavior.

Going a little farther in this direction, there is a close relation between $g_3$ and the more general class of $k$-fold unimodal rank generating functions. Let us define the {\it rank generating function $\widetilde{U}_k(z,q)$ for unimodal sequences with a $k$-fold peak} by the series
\begin{equation}\label{U_r}
\widetilde{U}_k(z,q):=\sum_{n=0}^{\infty}\frac{q^{kn}}{(zq;q)_n (z^{-1}q;q)_{n}}= \sum_{n=0}^{\infty}\sum_{m=-\infty}^{\infty}\widetilde{u}_k(m,n)z^m q^n,
\end{equation}
where $\widetilde{u}_k(m,n)$ is the number of $k$-fold peak unimodal sequences of rank $m$ and weight $n$. This identity follows directly from the combinatorial definition of $\widetilde{U}_k$, as Lovejoy noted to the author 
\cite{Lovejoy_private}: the $(z^{-1}q;q)_n^{-1} $ and $(zq;q)_n^{-1}$ generate the $a_i,b_j$ just as in (\ref{U}), and $q^{kn}$ inserts $k$ copies of $n$ as the $k$-fold peak.  

Then it is not hard to find (see Theorem 1.1 of \cite{Robert_zeta}) relations like 
\begin{equation}\label{U_1, U_2 modular}
\frac{1}{(zq;q)_{\infty}(z^{-1}q;q)_{\infty}}=2-z-z^{-1}+(z+z^{-1})\widetilde{U}_1(z,q)-\widetilde{U}_2(z,q),
\end{equation}
which of course is equal to $(1-z)\left<j_z\right>_q$ and is modular for $z=\zeta_*$, up to multiplication by a power of $q$. 
For example, noting that $z+z^{-1}=0$ when $z=i$, then (\ref{U_1, U_2 modular}) yields  
\begin{equation}\label{strongU_1, U_2 modular.rmk1}
2-\widetilde{U}_2(i,q)=(iq;q)_{\infty}^{-1}(-iq;q)_{\infty}^{-1}=(-q^2;q^2)_{\infty}^{-1},
\end{equation}
where $(-q^2;q^2)_{\infty}$ is essentially a modular function.

At this point we can compare (\ref{U_1, U_2 modular}) to Theorem \ref{theorem1}(\ref{thm1.1}) 
to solve for $g_3(z^{-1},q^{-1})$ in terms of $\widetilde{U}_1,\widetilde{U}_2$, but it is a little messy. 
However, it follows from a convenient rewriting of the right-hand side of Theorem \ref{theorem1}(\ref{umtfsum}) using geometric series
\begin{equation*}
\sum_{n=0}^{\infty}\frac{z}{(z;q)_{n+1} (z^{-1}q;q)_{n}}\left( \frac{z^{-1}q^{n+1}}{1-z^{-1}q^{n+1}}\right)=\frac{z}{1-z} \sum_{n=0}^{\infty}\sum_{k=1}^{\infty}\frac{z^{-k}q^{k(n+1)}}{(zq;q)_n (z^{-1}q;q)_n}
\end{equation*}
which converges absolutely for $|q|<|z|$, and then swapping order of summation, that in fact $g_3(z^{-1},q^{-1})$ can be written nicely in terms of the $\widetilde{U}_k$.

\begin{corollary}\label{another_cor}
For $|q|<1<|z|$, we have
\begin{equation*}
g_3(z^{-1},q^{-1})=\frac{z}{1-z}\sum_{k=1}^{\infty}\widetilde{U}_{k}(z,q) z^{-k}q^k.
\end{equation*}
\end{corollary}


Thus the inverted universal mock theta function leads to a type of two-variable generating function for the sequence of rank generating functions for unimodal sequences with $k$-fold peaks, $k=1,2,3,...$. 


\begin{proof} [Proof of Theorem \ref{theorem1}]
We begin by noting for $|q|<1, z\neq0$, 
\begin{equation*}
\left<j_z\right>_q  = (z;q)_{\infty}^{-1} (z^{-1} q; q)_{\infty}^{-1} =\prod_{n=0}^{\infty}\left(1-q^n (z+z^{-1}q-q^{n+1})\right)^{-1},
\end{equation*} 
where in the final step we multiplied together the $n$th terms from each $q$-Pochhammer symbol. Thus we have 
\begin{equation}\label{equivalent}
\prod_{n=0}^{\infty}\left(1-q^n (z+z^{-1}q-q^{n+1})\right)^{-1}=1+\sum_{n=1}^{\infty}\frac{q^n (z+z^{-1}q-q^{n+1})}{\prod_{j=0}^{n-1}\left(1-q^j (z+z^{-1}q-q^{j+1})\right)},
\end{equation}
which is easily seen to be absolutely convergent, and can be shown by expanding the product on the left as the telescoping series 
\begin{equation}\label{telescoping}
1+\sum_{n=1}^{\infty}\left(\frac{1}{\prod_{i=0}^{n}\left(1-q^i (z+z^{-1}q-q^i)\right)}-\frac{1}{\prod_{i=0}^{n-1}\left(1-q^{i-1} (z+z^{-1}q-q^{i-1})\right)}  \right)
\end{equation}
with a little arithmetic (for more details see the proof of Theorem 1.1 (2) in \cite{Robert_zeta}). Now, by the above considerations, (\ref{equivalent}) is equivalent to the following relation.

\begin{lemma}\label{otherlemma} 
For $|q|<1, z\neq 0$, we have
\begin{equation*}
\left<j_z\right>_q=1+(z+z^{-1}q)\sum_{n=1}^{\infty}\frac{q^n}{(z;q)_n (z^{-1}q;q)_n}-q\sum_{n=1}^{\infty}\frac{q^{2n}}{(z;q)_n (z^{-1}q;q)_n}.
\end{equation*}
\end{lemma}

We cannot help but notice how 
both series on the right-hand side of Lemma \ref{otherlemma} resemble the right-hand summation of identity (\ref{U}) for $\widetilde{U}(z,q)$. This is not a coincidence; it follows right away from the simple observation 
\begin{equation*}
\widetilde{U}(z,q)=\sum_{n=0}^{\infty}\frac{q^n}{(zq;q)_n (z^{-1}q;q)_n}=q^{-1}(1-z)\sum_{n=0}^{\infty}\frac{q^{n+1}(1-z^{-1}q^{n+1})}{(z;q)_{n+1} (z^{-1}q;q)_{n+1}},
\end{equation*}
that $\widetilde{U}$ splits off in a very similar fashion 
to $\left<j_z\right>_q$ in Lemma \ref{otherlemma}, after taking into account $q\neq 0$:
\begin{equation}\label{other2}
\widetilde{U}(z,q)=q^{-1}(1-z)\sum_{n=1}^{\infty}\frac{q^n}{(z;q)_n (z^{-1}q;q)_n}-(zq)^{-1}(1-z)\sum_{n=1}^{\infty}\frac{q^{2n}}{(z;q)_n (z^{-1}q;q)_n}.
\end{equation}
Comparing Lemma \ref{otherlemma} and (\ref{other2}), 
plus a little bit of algebra, then gives 
\begin{equation}\label{midway}
\left<j_z\right>_q=1\  +\  \left[z(1-q)+z^{-1}q\right]\sum_{n=1}^{\infty}\frac{q^n}{(z;q)_n (z^{-1}q;q)_n}\  +\  \frac{zq^2}{1-z}\widetilde{U}(z,q).
\end{equation}

%

Now, to connect the remaining summation in (\ref{midway}) to the universal mock theta function $g_3$, we apply a somewhat clever factorization strategy in the $q$-Pochhammer symbols to arrive at a useful identity (see \cite{GR}, Appendix 1 (I.3)): 
\begin{flalign}\label{invertmtf}
\begin{split}
(z;q)_n (z^{-1}q;q)_n
&=\prod_{j=0}^{n-1} \left[(-zq^j)(1-z^{-1}(q^{-1})^{j})\right] \left[(-z^{-1}q^{j+1})(1-z(q^{-1})^{j+1})\right]\\
&=q^{n^2}(z^{-1};q^{-1})_n (zq^{-1};q^{-1})_n.\\
\end{split}
\end{flalign}
Thus 
\begin{flalign}\label{invert2}
\sum_{n=1}^{\infty}\frac{q^n}{(z;q)_n (z^{-1}q;q)_n}&=\sum_{n=1}^{\infty}\frac{q^n}{q^{n^2}(z^{-1};q^{-1})_n (zq^{-1};q^{-1})_n}\\&=\sum_{n=1}^{\infty}\frac{(q^{-1})^{n(n-1)}}{(z^{-1};q^{-1})_n (zq^{-1};q^{-1})_n}.
\end{flalign}
The right-hand side of (\ref{invert2}) is $g_3(z^{-1},q^{-1})$, noting that it converges under the same conditions as the left side (being merely a term-wise rewriting), but with $q=0$ omitted from the domain.

\begin{remark}
Equivalently, identities like these result from the observation that
\begin{equation*}
(1-zq^i)(1-z^{-1}q^{-i})^{-1}=-zq^i.
\end{equation*}
Taking the product over $0\leq i \leq n-1$ gives 
\begin{equation*}
(z;q)_n (z^{-1};q^{-1})_n^{-1}=(-1)^n z^n q^{n(n-1)/2}
\end{equation*} and, proceeding in this manner, a variety of $q$-series summand forms can be produced (and inverted as above) by creative manipulation.  
\end{remark}
\end{proof}

\begin{remark}
We note in passing that, using (7.2) and (8.2) of Fine \cite{Fine}, Ch. 1, together with Theorem \ref{theorem1}(\ref{umtfsum}), we can also write
\begin{flalign}
\begin{split}
g_3(z^{-1},q^{-1})\  =\  (z^{-1}q;q)_{\infty}^{-1}(-z;q)_{\infty}^{-1}\sum_{n=0}^{\infty} (-1)^nz^{-2n}q^{\frac{n(n+1)}{2}}-\sum_{n=0}^{\infty} z^{-n+1}(z^{-1};q)_{n}.\\
\end{split}
\end{flalign}
\end{remark}

Recall that many modular forms arise as specializations of $j(z;q)$ (because $j(z;q)$ is essentially a Jacobi form, see \cite{BFOR}), 
and that $g_3(z,q)$ is the prototype for the class of mock modular forms that (to slightly abuse Ramanujan's words) 
``enter into mathematics as beautifully'' as the modular cases \cite{Hardy}. It is interesting that these important number-theoretic objects which are speculatively associated in the literature 
to opposite extremes of the universe---subatomic and supermassive---are themselves intertwined via the $q$-bracket from statistical physics, which applies to phenomena at every scale.
\  
\  
\

\section{Approaching roots of unity radially from within (and without)}\label{Sect2}

One point that arises in (\ref{invertmtf}) and (\ref{invert2}) above is that, evidently, one can construct pairs of $q$-series $\varphi_1(q), \varphi_2(q)$, convergent for $|q|<1$, with the property   
\begin{equation}\label{inversion}
\varphi_1(q)=\varphi_2(q^{-1})
\end{equation} 
({thus $\varphi_1(q)+\varphi_2(q),\  \varphi_1(q)\varphi_2(q)$ are self-reciprocal\footnote{See \cite{LectureHall}}}). This type of phenomenon, relating functions inside and outside the unit disk, is studied in \cite{BFR1,Folsom}. In particular, the universal mock theta function $g_3$ can be written 
as a piecewise function
\begin{equation}\label{umtfquantum}
g_3(z,q)= \left\{
        \begin{array}{ll}
            \sum_{n=1}^{\infty}\frac{q^{n(n-1)}}{(z;q)_n (z^{-1}q;q)_n} & \text{if}\  |q|<1,\\
            \  \\  
            \sum_{n=1}^{\infty}\frac{(q^{-1})^n}{(z;q^{-1})_n (z^{-1}q^{-1};q^{-1})_n} & \text{if}\  |q|>1,
        \end{array}
    \right.
\end{equation} 
for $q$ inside or outside the unit circle, respectively, and $z\neq 0\  \text{or}\  1$. What of $g_3(z,q)$ for $q$ lying {\it on} the circle? Generically one expects this question to be somewhat dicey. 

To be precise in what follows, for $\zeta$ on the unit circle we define $g_3(z,\zeta)$ to mean the limit of $g_3(z,q)$ as $q\to \zeta$  radially from within (or without if the context allows), when the limit exists. Recalling the notation $\zeta_m:=e^{2\pi \text{i}/m}$, it turns out that for $\zeta=\zeta_*$ an appropriate root of unity, $g_3(z,\zeta_*)$ is finite, both in value and length of the sum. 

\begin{theorem}\label{theorem2}
For $q=\zeta_m$ a primitive $m$th root of unity, $z\neq 0, 1$, or a rational power of 
$\zeta_m$, and $z^m+z^{-m}\neq 1$, the odd-order universal mock theta function is given by the finite formula
\begin{equation*}
g_3(z,\zeta_m)\  =\  (1-z^m-z^{-m})^{-1}\  \sum_{n=0}^{m-1}\zeta_m^{n}\  (z;\zeta_m)_{n}(z^{-1}\zeta_m;\zeta_m)_{n}.
\end{equation*}
\end{theorem}

\begin{remark}
Bringmann--Rolen \cite{BringmannRolen} and Jang--L\"{o}brich \cite{JangLobrich} have studied radial limits of universal mock theta functions from other perspectives. 
\end{remark}

Thus, under the right conditions, (\ref{umtfquantum}) together with Theorem \ref{theorem2} suggest $g_3(z,q)$ can, in a certain sense, ``pass through'' the unit circle at roots of unity  (as a function of $q$ following a radial path) into the complex plane beyond, and vice versa, while always remaining finite. 

In the theory of quantum modular forms, one encounters functions that exhibit similar behavior (see \cite{BFOR, RolenSchneider}).


\begin{definition}\label{qmfdef}
Following Zagier \cite{Zagierquantum}, we say a function $f: \mathbb{P}^1(\mathbb Q) \backslash S \to \mathbb C$, for a discrete subset $S$, is a quantum modular form if $f(x)-f|_k\gamma(x)=h_{\gamma}(x)$ for a ``suitably nice'' function $h_{\gamma}(x)$, for any $\gamma \in \Gamma$ a congruence subgroup of $\operatorname{SL}_2(\mathbb Z)$, $|_k$ is the usual Petersson slash operator (see \cite{Ono_web}), and ``suitably nice'' implies some pertinent analyticity condition, e.g. $\mathcal{C}_k,\mathcal{C}_{\infty}$, etc. 
\end{definition}
%

These are functions that, in addition to being ``almost'' modular, generically ``blow up'' as $q$ approaches the unit circle from within, 
but are finite 
when $q$ radially approaches certain roots of unity or other isolated points, in which case the limiting values have been related to special values of $L$-functions \cite{BFOR}---and might even extend 
to the complex plane beyond the unit circle as in (\ref{umtfquantum}), a phenomenon called {\it renormalization} (see \cite{Rhoades}). 
We see that $g_3$ exhibits this type of 
renormalization behavior.

Some mock theta functions are closely related to quantum modular forms. As we noted in Section \ref{Sect0}, Ramanujan's mock theta function $f(q)$ (from eq. (\ref{f(q)def})) is, at even-order roots of unity, essentially a quantum modular form plus a modular form\footnote{We give examples of similar cases in \cite{Schneider_strange}.}, through its relation to another rank generating function, the {\it rank generating function $U(z,q)$ for strongly unimodal sequences} \cite{BOPR, FolsomOnoRhoades}, defined by  
\begin{equation}\label{stronglyU}
U(z,q):=\sum_{n=0}^{\infty}q^{n+1}(-zq;q)_n (-z^{-1}q;q)_n = \sum_{n=0}^{\infty}\sum_{m=-\infty}^{\infty}{u}(m,n)z^m q^n,
\end{equation}
where ${u}(m,n)$ is the number of strongly unimodal sequences of rank $m$ and weight $n$. As with $\widetilde{U}, \widetilde{U}_k$ previously, the identity follows directly from the combinatorial definition: here, the $(-z^{-1}q;q)_n^{-1}$ and $(-zq;q)_n^{-1}$ generate {\it distinct} $a_i\leq n,b_j \leq n$, respectively, and $q^{n+1}$ inserts $n+1$ as the peak term. 

This $U(z,q)$ is a function that strikes deep: up to multiplication by rational powers of $q$, $U(i,q)$ is mock modular, $U(1,q)$ is mixed mock modular, and $U(-1,q)$ is a quantum modular form that can be completed to yield a weight $3/2$ non-holomorphic modular form \cite{BFOR}; in fact, mock and quantum modular properties of $U(z,q)$ are proved in generality for $z$ in an infinite set of roots of unity in \cite{FKTVY}.


Of course, $U$ is the $k=1$ case of the {\it rank generating function $U_k(z,q)$ for strongly unimodal sequences with $k$-fold peak}, given by

\begin{equation}
U_k(z,q):=\sum_{n=0}^{\infty}q^{k(n+1)}(-zq;q)_n (-z^{-1}q;q)_{n}= \sum_{n=0}^{\infty}\sum_{m=-\infty}^{\infty}{u}_k(m,n)z^m q^n,
\end{equation}
where ${u}_k(m,n)$ counts $k$-fold peak strongly unimodal sequences of rank $m$ and weight $n$, as above. 
Once again, we note the symmetry $U_k(z^{-1},q)=U_k(z,q)$. As with $\widetilde{U}_k$ in (\ref{U_1, U_2 modular}), we can find (see Theorem 2.8 of \cite{Robert_zeta}) nice relations like
\begin{equation}\label{strongU_1, U_2 modular2}
(zq;q)_{\infty}(z^{-1}q;q)_{\infty}=1-(z+z^{-1}){U}_1(z,q)+{U}_2(z,q),
\end{equation}
which is modular for $z=\zeta_*$ a root of unity, up to multiplication by a power of $q$. 
For instance, at $z=i$, 
equation (\ref{strongU_1, U_2 modular2}) gives 
\begin{equation}\label{strongU_1, U_2 modular.rmk2}
1+{U}_2(i,q)=(iq;q)_{\infty}(-iq;q)_{\infty}=(-q^2;q^2)_{\infty}.
\end{equation}

\begin{remark}
Multiplying (\ref{strongU_1, U_2 modular.rmk1}) and (\ref{strongU_1, U_2 modular.rmk2}) leads to a nice pair of identities relating $U_2$ and $\widetilde{U}_2$:
\begin{equation}
U_2(i,q)=\frac{1-\widetilde{U}_2(i,q)}{\widetilde{U}_2(i,q)-2},\  \  \  \widetilde{U}_2(i,q)=\frac{1+2U_2(i,q)}{1+U_2(i,q)}.
\end{equation}
\end{remark}

Now, taking a similar approach to that in Section \ref{Sect1} with regard to $\widetilde{U}_k$, we can find from Theorem \ref{theorem2}, using an evaluation of $U_k(-z,q)$ at $q=\zeta_m$ much like the theorem\footnote{We note for $k=1, z=1$, $m$ even, that the summation in (\ref{strongU_rfinite}) appears in the right-hand side of (\ref{Watson2}).}
\begin{equation}\label{strongU_rfinite}
U_k(-z,\zeta_m)=\frac{-1}{1-z^m-z^{-m}}\sum_{n=0}^{m-1}\zeta_m^{k(n+1)}(z\zeta_m;\zeta_m)_n (z^{-1}\zeta_m;\zeta_m)_{n},
\end{equation}
that the universal mock theta function $g_3$ also connects to these rank generating functions $U_k$ at roots of unity, through a similar relation to Corollary \ref{another_cor}. 

\begin{corollary}\label{cor2}
For $|z|<1$, we have
\begin{equation*}
g_3(z,\zeta_m)=\frac{z-1}{z}\sum_{k=1}^{\infty} {U}_{k}(-z,\zeta_m)z^{k}\zeta_m^{-k}.
\end{equation*}
\end{corollary}


%
%

How suggestive it is, in light of the 
relationship between $f(q)$ and $U(-1,q)$ \cite{FolsomOnoRhoades}, to see specializations of $g_3$ giving rise to both forms of $k$-fold unimodal rank generating functions in Corollaries \ref{another_cor} and \ref{cor2}. 

\begin{proof}[Proof of Theorem \ref{theorem2} and Corollary \ref{cor2}]
We start with an elementary observation. For an arbitrary $q$-series with 
coefficients $d_n$, then in the limit as $q$ approaches an $m$th root of unity $\zeta_m$ radially from within the unit circle, we have
\begin{equation}\label{finite}
\lim_{q\to \zeta_m} \sum_{n=1}^{\infty}d_n q^n=\sum_{n=1}^{m}D_n \zeta_m^n\  \text{where}\  D_n:=\sum_{j=0}^{\infty}d_{n+mj},
\end{equation}
so long as $\sum_j d_{n+mj}$ converges. The moral of this example: $q$-series {\it want} to be finite at roots of unity. 

In a similar direction, Theorem \ref{theorem2} arises from the following very general lemma, which the author has spoken on 
at Emory University since 2012 and used for heuristics, 
but has not published previously. It is really Lemma \ref{lemma} below that is the pivotal result of Section \ref{Sect2}; the applications to $g_3(z,\zeta_m)$ form an interesting exercise. 

\begin{lemma}\label{lemma}
Suppose $\phi: \mathbb Z^+ \to \mathbb C$ is a periodic function of period $m\in\mathbb Z^+$, i.e., $\phi(r+mk)=\phi(r)$ for all $k\in \mathbb Z$. Define $f: \mathbb Z^+ \to \mathbb C$ by 
the product
\begin{equation*}
f(j):=\prod_{i=1}^{j}\phi(i),
\end{equation*}
and its summatory function $F(n)$ by $F(0):=0$ and, for $n\geq 1$,
\begin{equation*}
F(n):=\sum_{j=1}^{n}f(j),\  \  \  \  F(\infty):=\lim_{n\to \infty}F(n)\  \text{if the limit exists}.
\end{equation*}
Then the following statements are true:
\begin{enumerate}[(i)]
\item \label{periodic3} For $0\leq r <m$ we have
\begin{equation*}
F(mk+r)=\frac{1-f(m)^k}{1-f(m)}F(m)+f(m)^k F(r).
\end{equation*}
\item \label{periodic4}
For $|f(m)|<1$ we have the finite formula
\begin{equation*}
F(\infty)=\frac{F(m)}{1-f(m)}.
\end{equation*}
\end{enumerate}
\end{lemma}  
  
\begin{proof}[Proof of Lemma \ref{lemma}]
First we observe that
\begin{equation}\label{periodic5}
f(mk)=\prod_{i=1}^{mk}\phi(i)=\left(\prod_{i=1}^{m}\phi(i)\right)^k=f(m)^k,
\end{equation}
by the periodicity of $\phi$. Then by the definition of $F(n)$ in Lemma \ref{lemma} together with (\ref{periodic5}) we can rewrite
\begin{flalign}
\begin{split}
&F(mk+r)\\&=\sum_{j=1}^{m}f(j)+\sum_{j=m+1}^{2m}f(j)+\sum_{j=2m+1}^{3m}f(j)+...+\sum_{j=m(k-1)+1}^{mk}f(j)+\sum_{j=mk+1}^{mk+r}f(j)\\
&=\left(1+f(m)+f(m)^2+...+f(m)^{k-1}\right)\sum_{j=1}^{m}f(j)+f(m)^k\sum_{j=1}^{r}f(j).\\
\end{split}
\end{flalign}
Recognizing the sum $1+f(m)+f(m)^2+...$ as a finite geometric series completes the proof of (\ref{periodic3}). If $|f(m)|<1$, the infinite case gives (\ref{periodic4}).
\end{proof}

\begin{remark}
Euler's continued fraction formula \cite{Euler} allows one to rewrite any hypergeometric sum as a continued fraction, and vice versa. Then we get another finite formula for $F(\infty)$, which holds for any convergent continued fraction of the following shape with periodic coefficients, including $q$-hypergeometric series when $q$ is replaced by appropriate $\zeta_m$: 
\begin{equation}
F(\infty)=
\frac{{\phi(1)}}{1-\frac{\phi(2)}{1+\phi(2)-\frac{\phi(3)}{1+\phi(3)-\frac{\phi(4)}{1+\phi(4)-...}}}}\  =\  \frac{1}{1-f(m)}\left(\frac{\phi(1)}{1-\frac{\phi(2)}{1+\phi(2)-\frac{\phi(3)}{1+...-\frac{\phi(m)}{1+\phi(m)}}}}\right).
\end{equation}
Therefore, the finiteness and renormalization considerations in this section also apply to $q$-continued fractions.
\end{remark}

Clearly if we take $\phi$ to be sine, cosine, etc. in Lemma \ref{lemma}, we can produce a variety of trigonometric identities. More pertinently, if we replace $\phi(i)$ with $\widetilde{\phi}(t,i):=  t \phi(i)$, this $\widetilde{\phi}$ also has period $m$; then we see $\widetilde{f}(j):=\prod_{i=1}^{j}\widetilde{\phi}(t,i)= t^j f(j)$. Thus the summatory functions $\widetilde{F}(n)=\widetilde{F}(t,n)$ and $\widetilde{F}(\infty)=\widetilde{F}(t,\infty)$ represent a polynomial and a power series in $t$, respectively---which are, respectively, subject to (\ref{periodic3}) and (\ref{periodic4}) of Lemma \ref{lemma}. Then for $\phi$ with period $k$ and the product $f$ as defined above, we get identities like
\begin{equation}
\sum_{n=1}^{\infty}f(n)t^n=\frac{1}{1-f(k)t^k}\sum_{n=1}^{k}f(n)t^n.
\end{equation} 
(We could also take $\widetilde{\phi}(t,i)$ equal to $t^i\phi(i)$ or $t^{2i}\phi(i)$ or $t^{2i-1}\phi(i)$, to lead to power series of other familiar shapes; however, such $\widetilde{\phi}$ are not generally periodic.)

Thinking along these lines, if  we set  
\begin{equation*}
\phi(i)=  t \frac{(1-a_1 q^{i-1})(1-a_2 q^{i-1})...(1-a_r q^{i-1})}{(1-b_1 q^{i-1})(1-b_2 q^{i-1})...(1-b_s q^{i-1})}
\end{equation*}
for $a_*, b_* \in \mathbb C$, the product $f(j)$ becomes a quotient of $q$-Pochhammer symbols, 
producing the $q$-hypergeometric series
$$F(t,\infty)=\  _{r}F_s (a_1,...,a_r;b_1,...,b_s; t:q).$$ If $q\to \zeta_m$ an $m$th root of unity, then $\phi$ is also cyclic of period $m$,
and in the radial limit $_{r}F_s (a_1,...,a_r;b_1,...,b_s; t:\zeta_m)$ is subject to Lemma \ref{lemma}(\ref{periodic4}), 
so long as in the denominator $(1-b_* \zeta_m^{i})\neq 0$ for any $i$.

Remembering the ``moral'' of equation (\ref{finite}), then similar considerations apply to almost all $q$-series and Eulerian series, for $q=\zeta_m$ a root of unity that does not produce singularities. In particular, so long as the choice of $z$ also does not lead to singularities, it is immediate from Lemma \ref{lemma} by the definition (\ref{umtf}) of $g_3$ that 
\begin{flalign}\label{simple}
\begin{split}
g_3(z,\zeta_m)&=\frac{1}{1-(z;\zeta_m)_m^{-1}(z^{-1}\zeta_m;\zeta_m)_m^{-1}}\sum_{n=1}^{m} \frac{\zeta_m^{n(n-1)}}{(z;\zeta_m)_{n}(z^{-1}\zeta_m;\zeta_m)_{n}}\\
&=\frac{2-z^m-z^{-m}}{1-z^m-z^{-m}}\sum_{n=1}^{m} \frac{\zeta_m^{n(n-1)}}{(z;\zeta_m)_{n}(z^{-1}\zeta_m;\zeta_m)_{n}},\\
\end{split}
\end{flalign}
where for the final equation we used the elementary fact that $$(X;\zeta_m)_m=1-X^m$$ in the leading factor. For a slightly simpler formula, we apply Lemma \ref{lemma} to the identity for $g_3(z^{-1};\zeta_m^{-1})$ in Theorem \ref{theorem1} instead, then take $z\mapsto z^{-1}$ and $\zeta_m\mapsto \zeta_m^{-1}$, to see
\begin{equation}\label{simpler}
g_3(z,\zeta_m)=\frac{2-z^m-z^{-m}}{1-z^m-z^{-m} }\sum_{n=1}^{m} \frac{\zeta_m^{-n}}{(z^{-1};\zeta_m^{-1})_{n}(z\zeta_m^{-1};\zeta_m^{-1})_{n}}.
\end{equation}
We note that the leading factor is symmetric under inversion of $z$.

\begin{remark} 
Jang--L\"{o}brich prove finite formulas similar 
to (\ref{simple}) and (\ref{simpler}) for $g_3(z,\zeta_m)$ \cite{JangLobrich}, by different methods.
\end{remark}

A particularly lovely aspect of $q$-series such as these is that they transform into an infinite menagerie of shapes, limited only by the curiosity of the analyst. (For instance, see Fine \cite{Fine} for a stunning exploration of $q$-hypergeometric series.\footnote{Fine writes: 
%
``The beauty and surprising nearness to the surface of some of the results could easily lead one to embark on an almost uncharted investigation of [one's] own.'' (\cite{Fine}, p. xi)}) 
Then a form like $g_3$ might have a number of different finite formulas. 

To derive Theorem \ref{theorem2}, which is simpler than the preceding expressions for $g_3$, we use another factorization strategy in the $q$-Pochhammer symbols. Again we exploit that $$(X;\zeta_m)_m=1-X^m=(X;\zeta_m^{-1})_{m};$$ thus for $0\leq n\leq m$, since $\zeta_m^{-j}=\zeta_m^{m-j}$ we have 
\begin{flalign}\label{factorization2}
\begin{split}
(X;\zeta_m^{-1})_n\  &=\  (1-X)(1-X\zeta_m^{m-1})(1-X\zeta_m^{m-2})...(1-X\zeta_m^{m-(n-1)})\\
&=\  \frac{(1-X)(X;\zeta_m)_m}{(X;\zeta_m)_{m-n+1}}\  =\  \frac{(1-X)(1-X^m)}{(X;\zeta_m)_{m-n+1}}.\\
\end{split}
\end{flalign}
Making the change of indices $n\mapsto m-n+1$ in the summation in (\ref{simpler}) then yields
\begin{equation*}
\sum_{n=1}^{m} \frac{\zeta_m^{-(m-n+1)}}{(z^{-1};\zeta_m^{-1})_{m-n+1}(z\zeta_m^{-1};\zeta_m^{-1})_{m-n+1}}=\sum_{n=1}^{m}\frac{\zeta_m^{n-1}(z^{-1};\zeta_m)_{n}(z\zeta_m^{-1};\zeta_m)_{n}}{(1-z^{-1})(1-z\zeta_m^{-1})(2-z^m-z^{-m})}.
\end{equation*}
Substituting this final expression back into (\ref{simpler}), with a little algebra and adjusting of indices, gives Theorem \ref{theorem2}. 

To prove Corollary \ref{cor2}, we use geometric series, along with an order-of-summation swap and index change, to rewrite Theorem \ref{theorem2} in the form
\begin{flalign}
\begin{split}
g_3(z,\zeta_m)&=\frac{1-z}{(1-z^m-z^{-m})}\sum_{n=0}^{m-1}\zeta_m^{n}\frac{(z\zeta_m;\zeta_m)_{n}(z^{-1}\zeta_m;\zeta_m)_{n}}{1-z\zeta_m^n}\\
&=\frac{z^{-1}(1-z)}{z(1-z^m-z^{-m})}\sum_{k=1}^{\infty}z^k\zeta_m^{-k}\sum_{n=0}^{m-1}\zeta_m^{k(n+1)}(\zeta_m;\zeta_m)_{n}(z^{-1}\zeta_m;\zeta_m)_{n}.\\
\end{split}
\end{flalign}
Comparing this with the formula (\ref{strongU_rfinite}) for $U(-z,\zeta_m)$, which follows easily from Lemma \ref{lemma}, gives the corollary. (The sum on the right might be simplified further using (\ref{finite}).)

\begin{remark}
Convergence in these formulas depends on one's choice of substitutions; 
for a particular choice, careful analysis may be required to show boundedness 
as $q$ approaches the 
natural boundary of a $q$-series 
(see Watson \cite{Watson} for examples). 
\end{remark}
\end{proof}

%

We note that a slight variation on the proof above leads to finite formulas at applicable roots of unity for the even-order universal mock theta function $g_2(z,q)$ of Gordon--McIntosh \cite{GordonMcIntosh} as well, by an alternative approach to that of Bringmann--Rolen \cite{BringmannRolen}. Using transformations from Andrews \cite{Andrews}, Fine \cite{Fine}, and other authors, still simpler formulas might be found for particular specializations of $g_3$ at roots of unity. We demonstrate this point below.

\begin{example}\label{example}
The limit of the mock theta function $f(q)$ at $\zeta_m$ an odd-order root of unity is given by
\begin{equation*}\label{f_finite}
f(\zeta_m)=1-\frac{2}{3}\sum_{n=1}^{m}(-1)^n \zeta_m^{-(n+1)} (-\zeta_m^{-1};\zeta_m^{-1})_n.
\end{equation*}
\end{example}

\begin{proof}[Proof of Example \ref{example}] The function $f(q)$ is convergent at odd roots of unity; however, for the reader's convenience, we will sketch a proof of convergence to the given value for just the case $q \to \zeta_m$ along a radial path. By (26.22) in \cite{Fine}, Ch. 3, Ramanujan's mock theta function $f(q)$ defined in (\ref{f(q)def}) 
can be rewritten
\begin{equation}\label{f(q)Fine}
f(q)=1-\sum_{n=1}^{\infty}\frac{(-1)^{n}q^n }{(-q;q)_n}.
\end{equation}
To show the summation on the right 
is bounded as $q$ approaches an odd-order root of unity radially, we exactly follow the steps of Watson's analysis of the mock theta function $f_0(q)$ in \cite{Watson}, Sec. 6. In Watson's nomenclature, take $M=2$, $N$ odd, to write $q=e^{2\pi \text{i}/N}=\zeta_N$. Then by replacing $q^{(nN+m)^2}$ with $(-1)^{nN+m}q^{nN+m}$ in the numerators of the $n\geq1$ terms of the series $f_0(q)$ 
(we note that Watson's $m$ is not the same as the subscript of $\zeta_m$ we use throughout this paper, which corresponds to $N$ in this proof), one sees 
$$\left|1- \sum_{n=1}^{\infty}\frac{(-1)^{n}q^n }{(-q;q)_n}\right|\leq 2\sum_{n=0}^{N-1}\left|\frac{q^n }{(-q;q)_n}\right|<\infty$$
when $q=\rho \zeta_N$ with $0\leq \rho \leq 1$. 
To see the value the series converges to, consider the $(Nk+r)$th partial sum, with $r<N$, of the right-hand side of (\ref{f(q)Fine}) as $\rho\to 1^-$, in light of Lemma \ref{lemma} (i). In fact, as $|(-1)^{N}\zeta_N^N/(-\zeta_N,\zeta_N)_N|=1/2<1$, then part (ii) of Lemma \ref{lemma} applies as $Nk+r\to\infty$ and (also taking into account that $f(\rho\zeta_N)$ converges uniformly for $\rho<1$) we may write 
\begin{equation}\label{rho_lim}
\lim_{\rho\to 1^-} \left(1- \sum_{n=1}^{\infty}\frac{(-1)^{n}(\rho\zeta_N)^n }{(-\rho\zeta_N;\rho\zeta_N)_n}\right)=1-\frac{2}{3}\sum_{n=1}^{N}\frac{(-1)^{n}\zeta_N^n }{(-\zeta_N;\zeta_N)_n}.
\end{equation}
Now, observe that (\ref{factorization2}) gives
\begin{equation}\label{factorization3}
(-\zeta_N;\zeta_N)_n^{-1}=(-\zeta_N^{-1};\zeta_N^{-1})_{N-n-1}.
\end{equation}
Applying (\ref{factorization3}) to the right-hand side of (\ref{rho_lim}), then making the change $n\mapsto N-n-1$ in the indices, 
we arrive at the desired result.
\end{proof}

Continuing in this fashion, we can find a formula for this radial limit that is even easier to compute.\footnote{In a 2013 study \cite{ClemmSchneider} of $f(q)$ at roots of unity using Example \ref{example2}, A. Clemm and the author found a number of elegant relations in SageMath, but without formal proof. For instance, at fifth-order roots of unity $\zeta_5^i$, one computes
$f(\zeta_5)f(\zeta_5^2)f(\zeta_5^3)f(\zeta_5^4)=256/81$.
Moreover, one computes 
\begin{equation*}
\zeta_5=9\slash 16\  f(\zeta_5)f(\zeta_5^3),\  \  \    
\zeta_5^2=9\slash 16\  f(\zeta_5)f(\zeta_5^2),\  \  \  
\zeta_5^3=9\slash 16\  f(\zeta_5^3)f(\zeta_5^4),\  \  \  
\zeta_5^4=9\slash 16\  f(\zeta_5^2)f(\zeta_5^4),
\end{equation*}
which are equivalent to 
$\zeta_5^i f(\zeta_5^i) = \zeta_5^{-i} f(\zeta_5^{-i}).$
Do similar relations hold for other roots of unity?} 

\begin{example}\label{example2}
For $\zeta_m$ an odd-order root of unity we have
\begin{equation*}\label{f_finite4}
f(\zeta_m)=\frac{4}{3}\sum_{n=1}^{m}(-1)^n (-\zeta_m^{-1};\zeta_m^{-1})_n.
\end{equation*}
\end{example}

\begin{proof}[Proof of Example \ref{example2}]
Here we use only finite sums, so we do not need to justify convergence. Let us define an auxiliary series
\begin{equation*}
h(\zeta_m)=\frac{2}{3}\sum_{n=1}^{m}(-1)^n (-\zeta_m^{-1};\zeta_m^{-1})_n.
\end{equation*}
Then using Example \ref{f_finite}, with a little arithmetic and adjusting of indices, gives
\begin{equation*}
f(\zeta_m)-h(\zeta_m)\  =\  1-\frac{2}{3}\sum_{n=1}^{m}(-1)^n (-\zeta_m^{-1};\zeta_m^{-1})_{n+1}\  =\  h(\zeta_m).
\end{equation*}
Comparing the left- and right-hand sides above implies our claim.
\end{proof}

Indeed, by the considerations here we can find both finite formulas at roots of unity, and inverted versions using factorizations such as in (\ref{invertmtf}) leading to forms such as (\ref{inversion}) and (\ref{umtfquantum}), for a great many $q$-hypergeometric series. Whether or not they enjoy modularity properties, these can display very interesting behaviors, 
emerging outside the unit circle radially from an entirely different point $\zeta_m^{-1}$ than the point on the circle $\zeta_m$ approached from within, and likewise when entering the circle radially at roots of unity from without, a little like quantum tunneling in physics. Moreover, the map inside the unit circle in the variable $q$ looks like an ``upside-down hyperbolic mirror-image'' of the function's behavior on the outside. 
(Taking $q \mapsto \overline{q}$ in either the $|q|<1$ or $|q|>1$ piece of (\ref{umtfquantum}) turns the map ``right-side up'', but at the expense of holomorphicity.\footnote{As Tyler Smith, Emory University Department of Physics, noted to the author \cite{TylerSmith}.}) 

This imagery reminds the author of depictions of white holes and wormholes in science fiction. Do there exist ``points-of-exit'' (and entry) analogous in some way to roots of unity, at the event horizon of a black hole? Is there a 
mirror-image universe 
contained within? 
We won't take these fantastical analogies too seriously, yet one is led to wonder: 
how deep is the connection between $q$-series and phenomena at nature's fringe?
 
\
\section*{Acknowledgments}
I am grateful to George Andrews for helpful comments pertaining to Examples \ref{example} and \ref{example2}, 
and for suggesting the Watson references; to Jeremy Lovejoy for useful correspondence regarding unimodal rank generating functions; to Larry Rolen and the anonymous referee for revisions that strengthened this work; and to Krishnaswami Alladi, Olivia Beckwith, David Borthwick, Amanda Clemm, John Duncan, Amanda Folsom, Marie Jameson, Benjamin Phelan, Robert Rhoades, Tyler Smith, and my Ph.D. advisor, Ken Ono, for discussions that informed my study.

\end{document}